\documentclass[11pt]{article}   
\usepackage{fullpage}
\usepackage{graphicx}

\usepackage{latexsym,amsfonts,amsmath,theorem,amssymb}
\usepackage{stmaryrd,array,tabularx,bbm}
\usepackage{pstricks,graphicx}
\usepackage{hyperref}
\definecolor{refcol}{rgb}{0,0,0.8}
\hypersetup{colorlinks}
\hypersetup{linkcolor=refcol, citecolor=refcol}

\usepackage{a4wide}

\usepackage{color}
\usepackage{caption}
\usepackage{theorem}
\usepackage{cleveref}
\usepackage{framed}
\usepackage{blindtext}
\usepackage{float}
\usepackage{todonotes}
\usepackage{bbm}
\usepackage[shortlabels]{enumitem}

\usepackage[caption=false]{subfig}
\usepackage{algorithm}
\usepackage{algorithmicx}
\usepackage[noend]{algpseudocode}
\usepackage{tikz}
\usetikzlibrary{arrows}
\usetikzlibrary{arrows.meta}
\usetikzlibrary{decorations.pathreplacing,angles,quotes}
\usepackage{caption}
\usepackage{scalerel}

\algrenewcommand\algorithmicrequire{\textbf{Input:}}
\algrenewcommand\algorithmicensure{\textbf{Output:}}
\newtheorem{theorem}{Theorem}[section]

\newtheorem{assumption}[theorem]{Assumption}
\newtheorem{definition}[theorem]{Definition}

\newenvironment{proof}{\begin{trivlist}
    \item[\hskip\labelsep{\bf Proof.}]}{$\hfill\Box$\end{trivlist}}

{\theoremstyle{plain} \theorembodyfont{\rmfamily}
\newtheorem{example}[theorem]{Example}
\newtheorem{remark}[theorem]{Remark}}

\numberwithin{equation}{section}
\numberwithin{figure}{section}
\numberwithin{table}{section}

\newcommand{\bsi}{{\boldsymbol{i}}}

\newcommand{\Nsub}{{N_{\rm sub}}}
\newcommand{\Nens}{{N_{\rm ens}}}

\newcommand{\Nobs}{{N_{\rm obs}}}
\newcommand{\comment}[1]{}

\newcommand{\spann}{{\mathrm{span}}}
\newcommand{\hPsi}{{\overline{\mathbf{F}}}}

\renewcommand{\hat}{\widehat}
\newcommand{\vertiii}[1]{{\left\vert\kern-0.25ex\left\vert\kern-0.25ex\left\vert #1 
\right\vert\kern-0.25ex\right\vert\kern-0.25ex\right\vert}}
\newcommand{\diff}{\mathrm{d}}

\definecolor{mh}{rgb}{0.1,0.45,0.1}

\definecolor{sw}{rgb}{0,0,1}

\usepackage{authblk}

\graphicspath{{./figures/}}

\title{Ensemble Kalman Inversion for Image Guided Guide Wire Navigation in Vascular Systems}
\author[1]{Matei Hanu}
\author[2]{Jürgen Hesser}
\author[3]{Guido Kanschat}
\author[2]{Javier Moviglia}
\author[1]{Claudia Schillings}
\author[2]{Jan Stallkamp}
\date{\today}
\affil[1]{\normalsize
  FU Berlin, Department of Mathematics and Computer Science\\
  14195 Berlin, Germany\\
\texttt{$\{$matei.hanu, c.schillings$\}$@fu-berlin.de}
}
\affil[2]{\normalsize
  Mannheim Institute for Intelligent Systems in Medicine, Heidelberg University, Mannheim, Germany\\
\texttt{$\{$juergen.hesser, javier.moviglia, jan.stallkamp$\}$@medma.uni-heidelberg.de}
}
\affil[3]{\normalsize
  Interdisciplinary Center for Scientific Computing, Heidelberg University, Heidelberg, Germany\\
\texttt{kanschat@uni-heidelberg.de}
}
\begin{document}

\maketitle

\begin{abstract}
This paper addresses the challenging task of guide wire navigation in cardiovascular interventions, focusing on the parameter estimation of a guide wire system using Ensemble Kalman Inversion (EKI) with a subsampling technique. The EKI uses an ensemble of particles to estimate the unknown quantities. However since the data misfit has to be computed for each particle in each iteration, the EKI may become computationally infeasible in the case of high-dimensional data, e.g. high-resolution images. This issue can been addressed by randomised algorithms that utilize only a random subset of the data in each iteration. We introduce and analyse a subsampling technique for the EKI, which is based on a continuous-time representation of stochastic gradient methods and apply it to on the parameter estimation of our guide wire system. Numerical experiments with real data from a simplified test setting demonstrate the potential of the method.
\end{abstract}

\section{Introduction} \label{sec:intro}
Vascular diseases account for approximately 40\% of all deaths in Germany
. The widespread use of minimally invasive interventions, such as percutaneous coronary angioplasty (PCA), is evident, with approximately 800,000 procedures performed in Germany in 2020 alone
. These interventions typically involve creating access to a larger artery, such as puncturing the arteria femoralis for femoral access. Subsequently, a guide wire is navigated under angiographic control by manipulating its distal end (extending from the patient) towards the targeted lesion, such as a coronary stenosis.

This navigation process is challenging and needs much experience for achieving a safe procedure. Inappropriate navigation poses risks, for complications such as dissections, which originate from excessive mechanical stress exerted on the vessel walls during the procedure. Therefore, robotic assistance in guide wire navigation is a very promising direction. Taking into account the aforementioned risks, it necessitates a system capable of executing navigation with a minimal amount of mechanical stress on vessel walls, ideally remaining below a user-defined threshold. 
Information regarding the vessel compliance and surface friction of vessel walls is usually unavailable, thus, parallel to the navigation process, it is necessary to learn them while introducing the wire by observing the wire shape over an angiographic imaging. Calculating uncertainties for these parameters and their motion enables the creation of a probabilistic model for assessing the mechanical stress during navigation, facilitating risk stratification.
We propose here a first step towards automated control of instruments, focusing on wire advancement in vascularities. In a simulation setting, utilizing force sensors, which enable precise force measurement, combined with high-resolution imaging, the behavior of the vascular environment is estimated and predicted in a simplified environment. A crucial aspect of the numerical approach is the real-time estimation and prediction of the guide wire position. We therefore focus on the Ensemble Kalman filter for inverse problems (EKI) with a subsampling technique. The EKI is known to work well in practice for the estimation of unknown parameters, particularly in settings close to the Gaussian case. The low computational costs and black-box use of the forward model make the EKI in this setting very appealing. To account for the large amount of data coming from the high resolution images, we suggest a subsampling technique.

\subsection{Literature overview} \label{subsec:literature}
Coronary vessels can be approximated by a rigid but timevariant tubular geometry responding to the heartbeat. For simplicity in algorithm development, we adopt a straightforward representation of the guide wire—a cylindrical rod with a specified elasticity constant—rather than a detailed guide wire model. Sharei et al \cite{Sharei2018} provide an overview of contemporary techniques for guidewire navigation, predominantly utilizing Finite Elements. These techniques include the Euler-Bernoulli beam model, the Kirchhoff rod model, and a Timoshenko beam, with a non-linear variant known as the Crosserat model. High-resolution images will be used to account for the simpler model. In order to estimate the unknown parameters from the model, we consider particle based methods, which have been demonstrated to perform well with low computational costs.
The Ensemble Kalman Filter (EnKF) is a common algorithm in solving inverse problems as well as data assimilation problems and was first introduced in \cite{Evensen2003}. This is due to its simple implementation and its stability when using small ensemble sizes \cite{Bergemann2009,Bergemann2010,Iglesias2014,Iglesias2016,Iglesias_2013,Li2009,Tong2015,Tong2016}. In recent time the continuous time limit of the Ensemble Kalman Inversion(EKI) has been getting more and more attention. And as a result, convergence analysis in the observation space has been formulated in \cite{Bloemker2019,Bloemker2021,Bungert2021,Schillings2017,Schillings2016}. In order to obtain convergence results in the parameter space, we usually have to employ some form of regularisation. In \cite{Tong2020}, for example, the authors analysed and discussed Tikhonov regularisation for the EKI, which will also be our choice of regularisation. Tikhonov regularisation has been getting further attention in the recent past. Analysis for adaptive Tikhonov strategies as well as in stochastic EKI setting has been done \cite{Weissmann2022}. We will also consider the mean-field limit of the EKI for our analysis \cite{Stuart2022,Ding2020}.\\
Due to the large computational costs of classical optimisation algorithms the stochastic gradient descent method has been introduced in \cite{Robbins1951}, and recently further analysed, e.g., \cite{Bottou2016,Nocedal2006}. Additionally to being computationally more efficient it is well suited for escaping local minimisers in non-convex optimisation problems \cite{Choromanska2014,Vidal2017}. For our subsampling scheme we will employ the analysis on the continuous-time limit of the stochastic gradient descent, which was first analysed in \cite{Latz2021}. This work has then been further generalised in \cite{Jin2021}. We will analyse subsampling in EKI for general nonlinear forward operators. An analysis for linear operators has been done in \cite{Hanu_2023}.

\subsection{Contributions and outline} \label{subsec:contributions}
This paper aims to advance robotic assistance for guide wire navigation in cardiovascular systems. We focus on a simplified setting and suggest a subsampling variant of the EKI to estimate the unknown parameters of the simulation and thus, improve the accuracy of predicitions. The key contributions of this article are as follows:
\begin{itemize}
    \item Formulation of a guide wire system with a feedback loop by high resolution images to learn the unknown parameter in the simulation.
    \item Adaption of a subsampling variant of the EKI to estimate the unknown parameters with low computational costs in a black-box setting.
    \item Analysis of a subsampling scheme for the EKI in the nonlinear setting.
    \item Numerical experiments with real data from a wire model.
\end{itemize}

The structure of the remaining article is as follows: In Section~\ref{sec:mathset}, we introduce the mathematical model of the guide wire system as well as formulate the inverse problem to reconstruct an image via our model. We describe the EKI and present results on convergence analysis in section \ref{sec:eki}; before introducing and analysing our subsampling scheme in section \ref{sec:subsampling}. We then proceed to show an numerical experiment in section \ref{sec:numeric} and complete our work with a conclusion in section \ref{sec:conclusion}.

\subsection{Notation} \label{subsec:notation}
We denote by $(\Omega, \mathcal{A}, \mathbb{P})$ a probability space. Let 
$X$ be a separable Hilbert space, and $Y := \mathbb{R}^\Nobs$  representing the \emph{parameter} and \emph{data spaces}, respectively. Here $\Nobs$ denotes the number of real-valued observation variables, 
We define inner products on $\mathbb{R}^n$ $\langle \cdot, \cdot \rangle$ and their associated Euclidean norms $\| \cdot \|$, or weighted inner products $\langle \cdot, \cdot \rangle_{\Gamma}:=\langle \Gamma^{-\frac12}\cdot, \Gamma^{-\frac12}\cdot \rangle$ and their corresponding weighted norms $\| \cdot \|_\Gamma:=\|\Gamma^{-1} \cdot \|$, where $\Gamma \in \mathbb{R}^{n \times n}$ is symmetric and positive definite matrix.

Furthermore, define the tensor product of vectors $x \in \mathbb{R}^n$ and $y \in \mathbb{R}^m$ as $x \otimes y := x y^\top$.

\section{Parameter estimation problem} \label{sec:mathset}

In an abstract setting, the aim is to recover the unknown parameter $u \in X$ from noisy observations $y\in Y$, i.e.
\begin{align}
    y=G(u)+ \eta, \label{eq_IP}
\end{align}
where $\eta \in Y$ is the additive observational noise, and $G: X \rightarrow Y$ is the solution operator of the underlying forward problem. We describe the model for the wire in the following subsection in more detail. Further, we assume that we can observe the state of the wire, i.e. the data $y \in Y := \mathbb{R}^\Nobs$, with $\Nobs \in \mathbb{N}$, corresponds to the image of the wire.

Within the Bayesian framework, we model $u$ and $\eta$ as independent random variables $u: \Omega \rightarrow X$ and $\eta: \Omega \rightarrow Y$, assuming $u \perp \eta$. Furthermore, assuming additive Gaussian noise $\eta \sim \mathrm{N}(0, \Gamma)$, the posterior distribution $\mu^y$ takes the form
\begin{equation}
    \mathrm d \mu^y(u) = \frac{1}{Z} \exp\left(-\frac{1}{2} \|y - Au\|_{\Gamma}^2\right) \mathrm d \mu_0(u),
\end{equation}
where $\mu_0$ represents the prior distribution on the unknown parameters and $Z = \mathbb E_{\mu_0} \exp\left(-\frac{1}{2} \|y - Au\|^2_{\Gamma} \right)$ denotes the normalization constant.

The primary focus lies in the computation of the Maximum a Posteriori (MAP) estimate, a point estimate for the unknown parameters. In the finite-dimensional setting, this corresponds to minimizing the negative log density. Assuming a Gaussian prior on the unknown parameters $u\sim\mathcal N(0, D_0)$, this results in minimizing
\begin{equation} \label{eq_pot_reg}
    \Phi^{\scaleto{\mathrm{reg}}{5pt}}(u):=\frac 12 \|y-G(u)\|_{\Gamma}^2+\frac{\alpha}{2} \|u\|^2_{D_0},
\end{equation}
where $\alpha>0$. We define $C_0=\alpha^{-1}D_0$ and define
\begin{equation}
\tilde G(u)=\begin{pmatrix}G(u) \\ C_0^{-\frac12}u\end{pmatrix},\quad  \tilde y=\begin{pmatrix}y \\ 0\end{pmatrix}\notag
\end{equation}
\eqref{eq_pot_reg} can equivalently formulated as
\begin{equation}\label{eq:minreg}
\Phi^{\scaleto{\mathrm{reg}}{5pt}}(u)=\frac 12 \|\tilde y-\tilde G(u)\|_{\Gamma}^2\,.
\end{equation}

\subsection{Wire model as forward model}\label{subsec:model}
To model the guide wire, we consider cosserat rods. Cosserat rods are used to simulate the dynamic behavior of filaments capable of bending, twisting, stretching, and shearing. The understanding of the dynamics of such objects is important since soft slender structures are prevalent in both natural and artificial systems. Examples include polymers, flagella, snakes, and space tethers. A detailed description of a straightforward and practical numerical implementation can be found in \cite{Gazzola_2018} and \cite{Zhang2019}. We will give a short introduction of the model and variables we consider, however we refer to the above mentioned articles for more details. Furthermore, we consider the mathematical software Pyelastica \cite{Tekinalp_2023}.
\subsubsection{Cosserat rods}
The major assumption for modelling Cosserat rods is that the rod length $L\in\mathbb{R}_{>0}$ is much larger then the radius $r\in\mathbb{R}_{>0}$, i.e., $(L\gg r)$.
The rod position is given by the centerline $r(s,t)$, where $s\in\left[0,L\right]$ describes a position on the rod and $t$ denotes the time. The position of the rod is given in the $3d$ Euclidean space.\\
We express vectors by the local (Lagrangian) frames $x=x_1d_1+x_2d_2+x_3d_3$, here $d_1$ and $d_2$ span the binormal plane and $d_3$ points along the centerline tangent. Furthermore, we define $Q=\{d_1,d_2,d_3\}$.\\

The cosserat rod dynamics at each cross section are described by the following differential equations:
\begin{itemize}[label={}]
    \item Linear momentum
        \begin{align*}
            \rho A \cdot \partial_t^2 \bar{r} = \partial_s\left(\frac{Q^TS\sigma}{e}\right)+e\bar{f}.
        \end{align*}\\
    \item Angular momentum
        \begin{align*}
            \frac{\rho I}{e}\partial_t\omega&=\partial_s\left(\frac{B\kappa}{e^3}+\frac{\kappa\times B\kappa}{e^3} \right)+\left(Q\frac{\bar{r}_s}{e}\times S\sigma\right)+\\
            &\left(\rho I\cdot\frac{\omega}{e}\right)\times \omega+\frac{\rho I \omega}{e^2}\cdot \partial_t e+ ec,
        \end{align*}
    \item And boundary conditions for position $r_i(t=0)=r_0$ and velocity $v_i(t=0)=v_0$.
\end{itemize}
The variables are: stretch ratio $e=\frac{\diff s}{\diff \hat{s}}\in \mathbb{R}$, where $s\in \mathbb{R}$ denotes the deformed configuration and $\hat{s}\in \mathbb{R}$ the reference configuration, cross-section area: $A=\frac{\bar{A}}{e}\in \mathbb{R}^{3\times 3}$, bending-stiffness matrix: $B=\frac{\bar{B}}{e^2}\in \mathbb{R}^{3\times 3}$, shearing-stiffness matrix: $S=\frac{\bar{S}}{e}\in \mathbb{R}^{3\times 3}$, second area moment of inertia: $I=\frac{\bar{I}}{e^2}\in \mathbb{R}^{3\times 3}$, local orientation of the rod $\bar{r}_s=e\bar{t}\in \mathbb{R}^{3}$, where $\bar{t}$ is a unit tangent vector, $E\in \mathbb{R}$ elastic Young's modulus, $G\in \mathbb{R}$ shear modulus, $I_i\in \mathbb{R} (i=1,2,3)$ second area moment of inertia, $\alpha_c=4/3$, external force $f\in \mathbb{R}^3$, couple line density $c\mathbb{R}^3$, mass per unit length $\rho\in \mathbb{R}$, curvature of the vector $\kappa\in \mathbb{R}^3$, angular velocity of the rod $\omega\in \mathbb{R}^3$, shear strain vector $\sigma\mathbb{R}^3$ and translational velocity $\bar{v}=\partial_t \bar{r} \mathbb{R}^3$. Here $B$ and $S$ depend on $G,E,I_i$ and $\alpha_c$.

\subsection{Experimental setup}\label{subsec:experiment}

In order to analyze and evaluate the physical behavior of the guide wire when applying a force in a static and dynamic state, two different setups have been carried out. In the first case, in a static analysis, a Radifocus Guide Wire M Standard type 0.035" guide wire (Terumo, Tokyo Japan) held in a vertical position against a surface (Fig. \ref{fig:static}) was used. A total of 47 images of the wire every 20mm of displacement have been photographed with a Full HD webcam camera. In this way it can observed its deformation due to its own weight for different positions.

\begin{figure}[H]
\centering
\captionsetup{width=.9\linewidth}
\includegraphics[scale=0.6,clip]{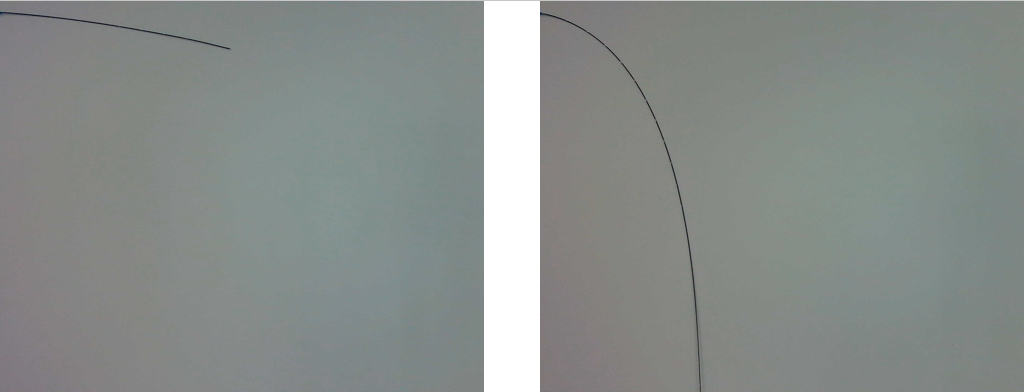}
\caption{Set up for static analysis. 200mm offset (left) 940mm offset (right).}
\label{fig:static}
\end{figure}

On the other hand, to evaluate the dynamic parameters of the wire, an electrical drive system has been designed. This consists of two degrees of freedom (axial and radial) and allows controlling both its position and its speed of movement. A diagram of the device built with its most representative components is found in Fig. \ref{fig:dynamic}. As can be seen, it consists of two motors, Motor 1 generates the axial movement and is contained within and fixed to a rotor. This itself is driven by Motor 2 and is what generates the radial movement. The prototype is controlled by a Raspberry PI 4 and was programmed using the framework ROS (Robot Operating System) for future sensor integrations and applications. The data extracted from the experiment were 80 images obtained every half second at an axial speed of 20mm/s through the phantom of a size of 100mm x 100mm shown in figure \ref{fig:dynamic}, which has three circles of different radii in different positions. This allows knowing the deformation of the wire in dynamic behavior. 

\begin{figure}[H]
\centering
\captionsetup{width=.9\linewidth}
\includegraphics[scale=0.75,clip]{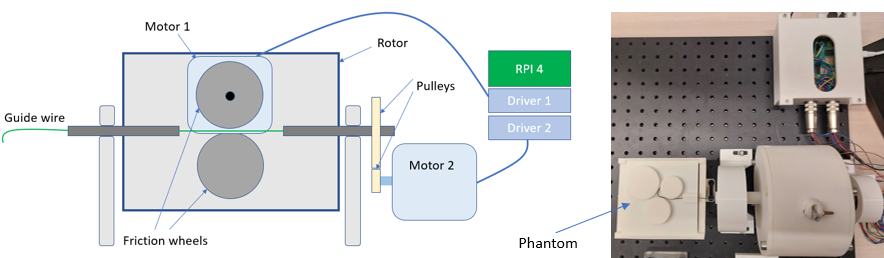}
\caption{Set up for dynamic analysis. Guided wire drive system diagram (left). Prototype of the drive system with phantom built in the laboratory to evaluate the guided wire in a dynamic state (right).}
\label{fig:dynamic}
\end{figure}

\subsection{Optimisation problem} \label{subsec:optim}

Our aim is to reproduce images that are obtained through the experimental setup described in section \ref{subsec:experiment} by solving the differntial equations described in section \ref{subsec:model}. We will consider two degrees of freedom, the density of the wire $\rho$ and the elastic Young modulus $E$ and assume that all other parameters are given. Hence, we focus on estimating the MAP of the inverse problem in order to solve \eqref{eq:minreg}, where $u=(\rho,E)^T\in\mathbb{R}^d$ with $d=2$.  
After solving the cosserat-rod ODEs, described in section \ref{subsec:model}, we need to perform additionally an image segmentation of the end position of the rod, Since the data $y$ comes in the form of an image.\\ We split the forward model into two steps
\begin{enumerate}
    \item Fix one end position of the rod and solve the cosserat rod model for the application of a constant force on the other end of the rod up until a fixed time $t_{end}$, given the parameter $u\in\mathbb{R}^d$.\\
    We define:
        \begin{align*}
            \mathcal{E}:\mathbb{R}^d&\to \mathbb{R}\times\mathbb{R}^3\\
            x&\mapsto \left(r(s,t),Q(s,t)\right),
        \end{align*}
        where $s\in\left[0,L\right], t\in\left[0,t_{end}\right], r(s,t)$ describes the center line at position $s$ and time $t$. And $Q(s,t)$ describes the orientation frame.
    \item Image segmentation: this part consists of the following steps
    \begin{enumerate}[(a)]
        \item We depict the last position of the rod as an image in 2D space, i.e., define
        \begin{align*}
            f_{img}:\mathbb{R}\times\mathbb{R}^3&\to \left(\mathbb{R}^{d_x\times d_y}\right)^3\\
            \left(r(s,t_{end}),Q(s,t_{end})\right)&\mapsto f_{img}\left(r(s,t_{end}),Q(s,t_{end})\right),
        \end{align*}
        here each element represents a pixel of the image, that consists of the RGB-values.
        \item Convert frames into greyscale image. Define
        \begin{align*}
            f_{RGBtoGREY}:\left(\mathbb{R}^{d_x\times d_y}\right)^3&\to \{0,1,...,255\}^{d_x\times d_y}\\
            x&\mapsto f_{RGBtoGREY}(x).
        \end{align*}
        $f_{RGBtoGREY}$ assigns each pixel, based on its shade, a number between $0$ and $255$, where $0$ represents the color black and $255$ the color white.
        \item Transform the greyscale images into binary images via a threshhold function, i.e., define
        \begin{align*}
            f_{thresh,\sigma}:\{0,1,...,255\}^{d_x\times d_y}&\to \{0,255\}^{d_x\times d_y}\\
            x_{i,j}&\to
                \begin{cases}
                    0,\quad &\text{if } x_{i,j}\leq\sigma\\
                    255,\quad &\text{if } x_{i,j}>\sigma
                \end{cases},
        \end{align*}
        where $\sigma\in\{1,2,...,255\}$ indicates whether pixels are turned black ($\left(f_{thresh,\sigma}(x)\right)_{i,j}=0$) or white ($\left(f_{thresh,\sigma}(x)\right)_{i,j}=255$).
        \item Distance transformation: Define 
        \begin{align*}
            f_{dist}:\{0,255\}^{d_x\times d_y}&\to \mathbb{R}^{d_x\cdot d_y}\\
            x&\mapsto f_{dist}(x).
        \end{align*}
        Here $f_{dist}$ computes the distance of each white pixel to the nearest black pixel. The exact method to compute the distance can be found in \cite{Felzenszwalb2012}. To conclude the segmentation we vectorize the image.
    \end{enumerate}
\end{enumerate}
    We denote the composition of the operators as $\mathcal{F}=f_{dist}\circ f_{thresh,\sigma}\circ f_{RGBtoGREY}\circ f_{img}$. Then $G:=\mathcal{F}\circ \mathcal{E}$, denotes the forward operator in potential \eqref{eq_pot_reg}. Note that for the minimisation of the potential we need to also apply distance transformation to the original image. The following example illustrates one example that we consider.

    \begin{example} \label{ex:main}
    Figure \ref{fig:example_nonshopped_andshopped} illustrates an example of an original taken image, where we apply some force to the rod and afterwards apply the image segmentation $\mathcal{F}$ to it. The end result is the distance transformation for this image, i.e. our observed data $y$ that we are trying to reconstruct. We illustrate the distance map of a toy example in Figure \ref{fig:1}
        \begin{figure}[H]
\centering
\captionsetup{width=.9\linewidth}
\includegraphics[scale=0.15, trim = 1cm -3cm 0cm 1.1cm,clip]{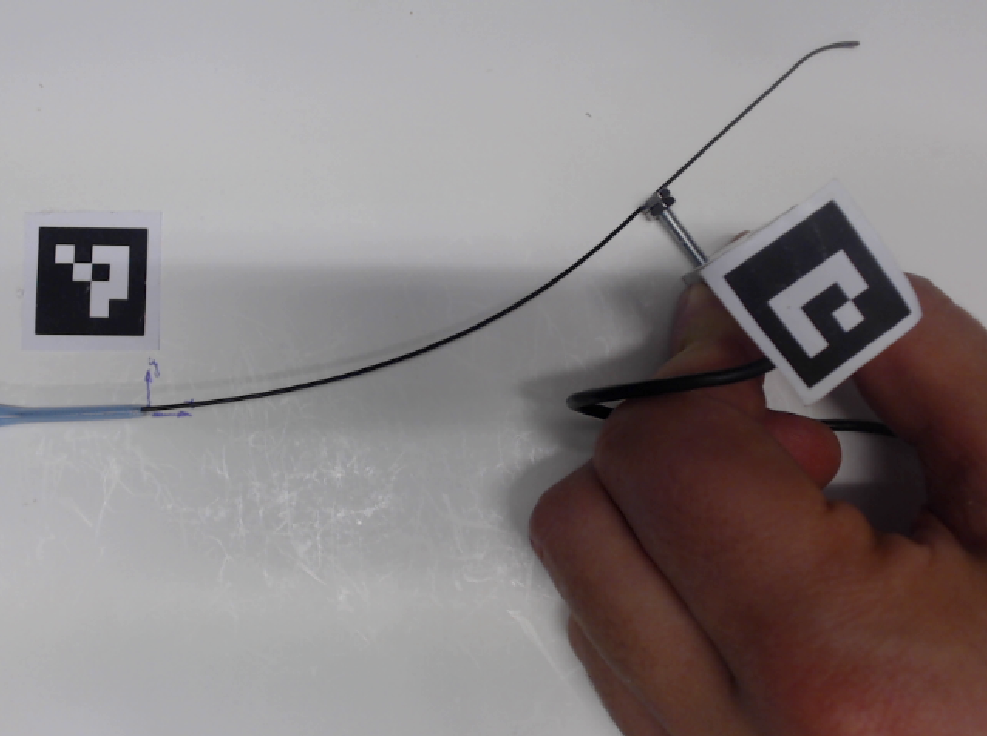}
\hspace*{0.3cm}%
\includegraphics[scale=0.4, trim = 0.5cm 0cm 0cm 1cm,clip]{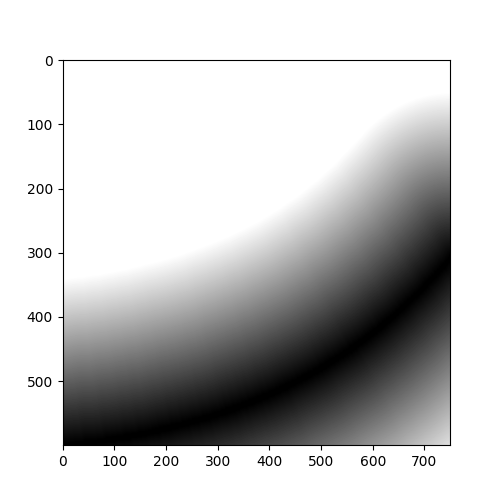}
\caption{Example of original taken image (left) with force application on to the rod and an image of the same rod after applying the transformation $\mathcal{F}$ to the original image.}
\label{fig:example_nonshopped_andshopped}
\end{figure}

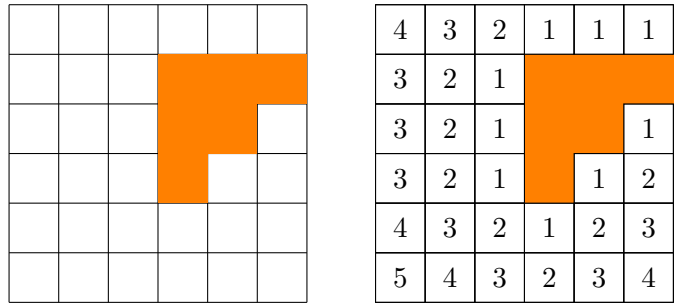
\begin{figure}[htbp]
\centering
\begin{tikzpicture}[scale=0.66]
    \foreach \y in {0,1,...,6} {
        \draw (0, \y) -- (6, \y);
    }
    
    \foreach \x in {0,1,...,6} {
        \draw (\x, 0) -- (\x, 6);
    }    
    \fill[orange] (3,4) rectangle (4,5);    
    \fill[orange] (3,3) rectangle (4,4);
    \fill[orange] (3,2) rectangle (4,3);
    \fill[orange] (4,4) rectangle (5,5);    
    \fill[orange] (4,3) rectangle (5,4);
    \fill[orange] (4,3) rectangle (5,4);
    \fill[orange] (5,4) rectangle (6,5);
\end{tikzpicture}\qquad
\begin{tikzpicture}[scale=0.66]
    \foreach \y in {0,1,...,6} {
        \draw (0, \y) -- (6, \y);
    }
    
    \foreach \x in {0,1,...,6} {
        \draw (\x, 0) -- (\x, 6);
    }    
    \fill[orange] (3,4) rectangle (4,5);    
    \fill[orange] (3,3) rectangle (4,4);
    \fill[orange] (3,2) rectangle (4,3);
    \fill[orange] (4,4) rectangle (5,5);    
    \fill[orange] (4,3) rectangle (5,4);
    \fill[orange] (4,3) rectangle (5,4);
    \fill[orange] (5,4) rectangle (6,5);

    \draw (0,0) rectangle (1,1) node[pos=.5,black] {$5$};
    \draw (1,0) rectangle (2,1) node[pos=.5,black] {$4$};
    \draw (2,0) rectangle (3,1) node[pos=.5,black] {$3$};
    \draw (3,0) rectangle (4,1) node[pos=.5,black] {$2$};
    \draw (4,0) rectangle (5,1) node[pos=.5,black] {$3$};
    \draw (5,0) rectangle (6,1) node[pos=.5,black] {$4$};
    \draw (0,1) rectangle (1,2) node[pos=.5,black] {$4$};
    \draw (1,1) rectangle (2,2) node[pos=.5,black] {$3$};
    \draw (2,1) rectangle (3,2) node[pos=.5,black] {$2$};
    \draw (3,1) rectangle (4,2) node[pos=.5,black] {$1$};
    \draw (4,1) rectangle (5,2) node[pos=.5,black] {$2$};
    \draw (5,1) rectangle (6,2) node[pos=.5,black] {$3$};
    \draw (0,2) rectangle (1,3) node[pos=.5,black] {$3$};
    \draw (1,2) rectangle (2,3) node[pos=.5,black] {$2$};
    \draw (2,2) rectangle (3,3) node[pos=.5,black] {$1$};
    \draw (4,2) rectangle (5,3) node[pos=.5,black] {$1$};
    \draw (5,2) rectangle (6,3) node[pos=.5,black] {$2$};
    \draw (0,3) rectangle (1,4) node[pos=.5,black] {$3$};
    \draw (1,3) rectangle (2,4) node[pos=.5,black] {$2$};
    \draw (2,3) rectangle (3,4) node[pos=.5,black] {$1$};
    \draw (5,3) rectangle (6,4) node[pos=.5,black] {$1$};
    \draw (0,4) rectangle (1,5) node[pos=.5,black] {$3$};
    \draw (1,4) rectangle (2,5) node[pos=.5,black] {$2$};
    \draw (2,4) rectangle (3,5) node[pos=.5,black] {$1$};
    \draw (0,5) rectangle (1,6) node[pos=.5,black] {$4$};
    \draw (1,5) rectangle (2,6) node[pos=.5,black] {$3$};
    \draw (2,5) rectangle (3,6) node[pos=.5,black] {$2$};
    \draw (3,5) rectangle (4,6) node[pos=.5,black] {$1$};
    \draw (4,5) rectangle (5,6) node[pos=.5,black] {$1$};
    \draw (5,5) rectangle (6,6) node[pos=.5,black] {$1$};

\end{tikzpicture}
\caption{Set (depicted in orange; left image) and its corresponding distance transformation (right image).}
\label{fig:1}
\end{figure}
    \end{example}

\section{Ensemble Kalman inversion (EKI)}\label{sec:eki}

We aim to find the minimiser of \eqref{eq:minreg} using the EKI, where we consider the continuous version of the EKI as introduced in  \cite{Schillings2017}.\\
We denote by $u_0 = (u_0^{(j)})_{j \in J} \in X^{\Nens}$ the initial ensemble, assuming, without loss of generality, that the family $(u_0^{(j)} - \bar{u}_0)_{j \in J}$ is linearly independent. Here, $\Nens \in \mathbb{N}$ with $\Nens \geq 2$, and $J := \{1, \ldots, \Nens\}$.

We consider an EKI approach without a stochastic component, here the particles are given by the solution of the following system of ODEs.

\begin{align} \label{EKI_basic}
    \frac{\mathrm{d} u^{(j)}(t)}{\mathrm{d}t} &= - \widehat{C}^{u G}_t \Gamma^{-1} (G(u^{(j)})(t) - y) \qquad (j \in J)\\
    u(0) &= u_0, \notag
\end{align}

where $\widehat{C}^{u G}_t$ is defined as

\[\widehat{C}^{u G}_t := \frac{1}{\Nens-1} \sum_{j = 1}^\Nens (u^{(j)}(t) - \overline{u}(t)) \otimes (G(u^{(j)})(t) - \overline{G(u)}(t)),\]

and $\overline{u}(t)$ is given by

\[\overline{G(u)}(t) = \frac{1}{\Nens}\sum_{j=1}^{\Nens}G(u^{(j)}) \qquad (t \geq 0).\]

It has been shown, that the continuous variant \eqref{EKI_basic} corresponds to a noise free limit for $t\to \infty$, i.e. overfitting will occur in the noisy case. Therefore, we include an additional regularisation term
\begin{align} \label{EKI_regul}
    \frac{\mathrm{d} u^{(j)}(t)}{\mathrm{d}t} &= - \widehat{C}^{u G}_t \Gamma^{-1} (G(u^{(j)})(t) - y)-\widehat{C}^{u}_t C_0^{-1}u^{(j)}_t \qquad (j \in J)\\
    u(0) &= u_0, \notag
\end{align}
with
$$
\widehat{C}^{u}_t := \frac{1}{\Nens-1} \sum_{j = 1}^\Nens (u^{(j)}(t) - \overline{u}(t)) \otimes (u^{(j)}(t) - \overline{u}(t)).
$$
In the nonlinear setting, variance inflation in form of a control of the mean in the observation space has been demonstrated to be crucial for the control of the nonlinearity, cf. \cite{Weissmann2022} for more details. The dynamics are then given by
\begin{align} \label{EKI_regul_vi}
    \frac{\mathrm{d} u^{(j)}(t)}{\mathrm{d}t} &= - \widehat{C}^{u G}_t\Gamma^{-1} (G(u_t^{(j)}) - y)-\widehat{C}^{u}_t C_0^{-1}u^{(j)}_t\notag\\
    &+\rho\widehat{C}^{u,G}_t \Gamma^{-1} ({G}(u_t^{(j)})-\bar G(u_t))+\rho \widehat{C}^{u}_t C_0^{-1}(u^{(j)}_t-\bar{u}_t) \qquad (j \in J)\\
    u(0) &= u_0, \notag
\end{align}
where $0\leq\rho< 1$. We can rewrite this as

\begin{align} \label{EKI_regul_vi_newnot}
    \frac{\mathrm{d} u^{(j)}(t)}{\mathrm{d}t} &= (1-\rho)\left[-\widehat{C}^{u G}_t\Gamma^{-1} (G(u_t^{(j)}) - y)-\widehat{C}^{u}_t C_0^{-1}u^{(j)}_t\right] \notag\\
    &+\rho\left[-\widehat{C}^{u G}_t\Gamma^{-1} (\bar G(u_t)- y)-\widehat{C}^{u}_t C_0^{-1}\bar{u}_t\right] \qquad (j \in J)\\
    u(0) &= u_0, \notag
\end{align}

\begin{remark}
    Note that, since $\Gamma$ is positive definite we can multiply \eqref{eq_IP} by the inverse square root of $\Gamma$ and not change the predicted solution.  For the sake of simplicity and without loss of generality, we will assume that $\Gamma$ is the identity matrix $\mathrm{Id}_k.$
\end{remark}

\subsection{Convergence analysis of EKI}
Based on the results in \cite{Weissmann2022} and assuming the convexity of the potentials, we summarize the results 
on the well-posedness and convergence properties of the EKI in the nonlinear setting as it follows: 

\begin{assumption} \label{assu:convex_lip}
    The functional $\Phi^{\scaleto{\mathrm{reg}}{5pt}}\in C^2(X,\mathbb{R}_+)$ satisfies
    \begin{enumerate}
        \item ($\mu$-strong convexity). There exists $\mu>0$ such that
            $$\Phi^{\scaleto{\mathrm{reg}}{5pt}}(x_1)-\Phi^{\scaleto{\mathrm{reg}}{5pt}}(x_2)\geq \langle \nabla \Phi^{\scaleto{\mathrm{reg}}{5pt}}(x_2),x_1-x_2\rangle+\frac{\mu}{2}\|x_1-x_2\|^2, \quad \forall\ x_1,x_2 \in X.$$
        \item ($L$-smoothness). There exists $L>0$ such that the gradient $\nabla \Phi^{\scaleto{\mathrm{reg}}{5pt}}$ satisfies:
        \[\|\nabla \Phi^{\scaleto{\mathrm{reg}}{5pt}}(x_1)-\nabla \Phi^{\scaleto{\mathrm{reg}}{5pt}}(x_2)\|\le L \|x_1-x_2\|\quad \forall\ x_1,x_2 \in X\,.  \]
    \end{enumerate}
\end{assumption}
Furthermore, we assume
\begin{assumption}\label{ass:linear_apprx}
    The forward operator $G\in C^2(X,Y)$ is locally Lipschitz continuous, with constant $c_{lip}>0$ and satisfies
        \begin{equation}\label{eqn:approx_error}
            G(x_1)=G(x_2)+DG(x_2)(x_1+x_2)+Res(x_1,x_2) \quad \forall x_1,x_2\in X,
        \end{equation}
        where $DG$ denotes the Fréchet derivative of $G$. Furthermore, the approximation error is bounded by
        \begin{equation}
            \|Res(x_1,x_2)\|_2\leq b_{res}\|x_1-x_2\|_2^2\,.
        \end{equation}
\end{assumption}

To understand the dynamics of the EKI, we start by discussing the subspace property, i.e. the EKI estimates are constrained to the span of the initial ensemble $S=\spann\{u^{(j)},j\in\{1,...,J\}\}$. To be more specific, the particles remain within the affine space $u_0^\perp +\mathcal E$ for all $t \ge 0$, as stated in \cite[Corollary 3.8]{Tong2020}. Here, $\mathcal{E}$ is defined as the span of vectors, i.e. $\mathcal{E}:=\{e^{(1)}(0), \ldots, e^{(\Nens)}(0)\}$, with $e^{(j)}=u^{(j)}-\bar u$, for $j\in\{1,\ldots,\Nens\}$ and $u_0^\perp=\bar u(0)-P_{\mathcal E} \bar u(0)$. We define $\mathcal{B}:=u_0^\perp +\mathcal E$.

\begin{theorem}[\protect{\cite{Weissmann2022}}]
Let $\{u^{(1)}(0),...,u^{(J)}(0)\}$ denote the initial ensemble. The ODE systems \eqref{EKI_regul} and \eqref{EKI_regul_vi}, admit unique global solutions $u^{(j)}(t)\in C^1([0,\infty);\mathcal{B})$ for all $j\in\{1,...,J\}$.
\end{theorem}

Therefore, the best possible solution that we can obtain through the EKI is given by best approximation in $u_0^\perp +\mathcal E$. We summarize in the following the convergence results for the various variants. 

\begin{theorem}[\protect{\cite{Weissmann2022}}]\label{thm:conver_simon}
    Suppose Assumptions \ref{assu:convex_lip} and \ref{ass:linear_apprx} are satisfied. Furthermore we define $V_e(t)=\frac{1}{J}\sum_{j=1}^J\frac{1}{2}\|e^{(j)}_t\|^2$, where $e^{(j)}_t=u^{(j)}_t-\bar{u}_t$. Let $j\in\{1,...,J\}$ and $u^{(j)}(t)$ be the solutions of \eqref{EKI_regul} (or respectively \eqref{EKI_regul_vi}) then it holds
    \begin{enumerate}
        \item The rate of the ensemble collapse is given by
        \[V_e(t) \in \mathcal O(t^{-1}).\]
        \item The smallest eigenvalue of the empirical covariance matrix $\widehat{C}^{u}_t$ remains strictly positive in the subspace $\mathcal{B}$, i.e. let
        \[\eta_0 = \min_{z\in B, \|z\|=1}\langle z,C(u_0),z\rangle >0.\]
        Then it holds for each $z\in\mathcal{B}$ with $\|z\|=1$
        \[ \langle z, \widehat{C}(u_t)z\rangle \ge \frac{1}{(1-\rho)mt+\eta_0},\]
        where $m>0$ depends on the eigenvalues of $\widehat{C}^{u}_0$ and $\Gamma$ and the Lipschitz constant $c_{lip}$ and $0\leq\rho<1$.
        \item Let $u^*$ be the unique minimiser of \eqref{eq:minreg} in $\mathcal{B}$ then it holds
        \[\frac{1}{J}\sum_{j=1}^J \Phi^{\scaleto{\mathrm{reg}}{5pt}}\left(u^{(j)}_t\right)-\Phi^{\scaleto{\mathrm{reg}}{5pt}}\left(u^*\right)\leq\left(\frac{c_1}{t+c_2}\right)^{\frac{1}{\alpha}}\]
        where $0<\alpha<(1-\rho)\frac{L}{\mu}(\sigma_{max}+c_{lip}\lambda_{max}\|C_0\|_{HS})$. Here $\sigma_{max}$ denotes the largest eigenvalue of $C_0^{-1}$, $\lambda_{max}$ denotes the larges eigenvalue of $\Gamma^{-1}$, $\|C_0\|_{HS}$ denotes the Hilbert-Schmidt norm, $0\leq\rho<1$ and $c_1,c_2>0$ depends on the constants from our assumptions \ref{assu:convex_lip} and \ref{ass:linear_apprx}.
    \end{enumerate}
\end{theorem}

\section{Subsampling in EKI} \label{sec:subsampling}
If $\Nobs$ is very large, as it will be the case for high-resolution images, it might be computationally infeasible to use the EKI framework to solve the inverse problem \eqref{eq_IP}. Therefore, we employ a subsampling strategy. This strategy involves splitting the data $y$ into $\Nsub$ subsets, denoted as $y_1,\ldots,y_\Nsub$, where $(y_1,\ldots,y_\Nsub) = y, \Nsub\in\{2,3,...\}$ and we use an index set $I := \{1,\ldots,\Nsub\}$ to represent these subsets. We define "data subspaces" $Y_1,\ldots, Y_\Nsub$ to accommodate this data partitioning, resulting in $Y := \prod_{i \in I} Y_i$. Regarding the noise we assume the existence of covariance matrices $\Gamma_i : Y_i \rightarrow Y_i$, for each $i \in I$, which collectively form a block-diagonal structure within $\Gamma$:
$$
\Gamma = \begin{pmatrix} \Gamma_1 & & & \\ & \Gamma_2 & & \\ & & \ddots & \\ & & & \Gamma_\Nsub \end{pmatrix}.
$$
Finally, we split the operator $G$ into a family of operators $(G_i)_{i \in I}$ obtaining the family of inverse problems

\begin{align*}
    G_1(u) + \eta_1 &= y_1 \notag\\
    \vdots& \\
     G_\Nsub(u) + \eta_\Nsub &= y_\Nsub \notag,
\end{align*}
where $\eta_i$ is a realisation of the a Gaussian random variable with zero-mean and covariance matrix $\Gamma_i$.

Similar to above we assume $\Gamma_i$ is the identity matrix $\mathrm{Id}_{ki}$ for the remaining discussion. Then for our analysis we consider the family of potentials
\begin{align*}
 \Phi_i(u) &:= \frac{1}{2} \|G_i(u) - y_i \|^2 \qquad (i \in I).
\end{align*}
When adding regularisation we can define the following entities to obtain a compact representation. We set $C_0=\frac{\Nsub}{\alpha}D_0$ and define
\begin{equation}
\tilde G_i(u)=\begin{pmatrix}G_i(u) \\ C_0^{-\frac12}u\end{pmatrix},\quad  \tilde y=\begin{pmatrix}y \\ 0\end{pmatrix}\notag
\end{equation}
Then we obtain the family of regularised potentials
\begin{equation}
\Phi_i^{\scaleto{\mathrm{reg}}{5pt}}(u)=\frac 12 \|\tilde y_i-\tilde G_i(u)\|^2,\quad (i \in I).\notag
\end{equation}
Note that we scale the regularisation parameter by $\Nsub^{-1}$ so that we obtain.
\begin{equation}
\Phi^{\scaleto{\mathrm{reg}}{5pt}}(u)=\sum_{i=1}^\Nsub\Phi_i^{\scaleto{\mathrm{reg}}{5pt}}(u).\notag
\end{equation}

 We will consider the right hand sides of the ODEs \eqref{EKI_regul} and \eqref{EKI_regul_vi} but replace $G$ with $G_i$ and respectively $y$ with $y_i$ randomly. This method was introduced in \cite{Hanu_2023} and is known as single-subsampling. We illustrate the idea in figure \ref{fig:cartoon_subsamp}.
 For the analysis of the subsampling scheme, we will assume the same convexity assumptions on the sub-potentials $\Phi_i$ as well as the same regularity assumptions on the family of forward operators $G_i$ for all $i\in I$.

\begin{assumption}\label{assu:subpotentials_gradients}
    The families of potentials $(\Phi_i,i\in I)$ and $(G_i,i\in I)$ satisfy assumptions \ref{assu:convex_lip} and \ref{ass:linear_apprx}.
\end{assumption}

\begin{figure}
\centering
\input{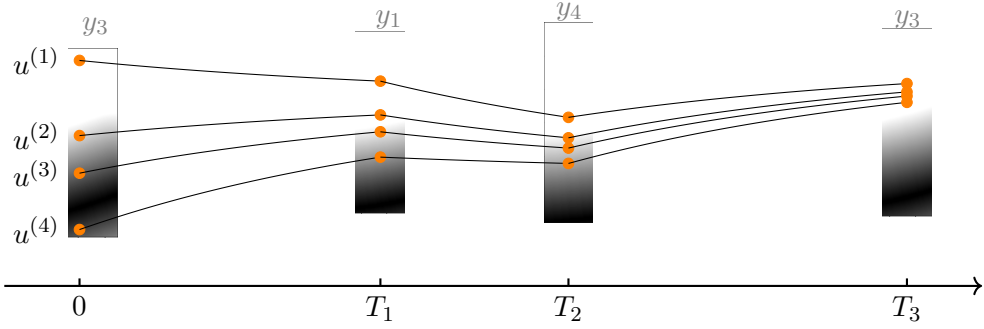}
\caption{Illustration of EKI with subsampling. Here we took the distance transformation from image \ref{fig:example_nonshopped_andshopped} and split the image horizontally into four subimages. Then at each time point where the data is changed, the particles see a different image until the next time when the data is changed.}
\label{fig:cartoon_subsamp}
\end{figure}

Then the flows that we are going to consider are given by
\begin{align} \label{EKI_regul_sub}
    \frac{\mathrm{d} u^{(j)}(t)}{\mathrm{d}t} &= - \widehat{C}^{u G}_t (G_{\bsi(t)}(u^{(j)})(t) - y_{\bsi(t)})-\widehat{C}^{u}_t C_0^{-1}u^{(j)}_t \qquad (j \in J)\\
    u(0) &= u_0, \notag
\end{align}
when we only consider regularisation. When we also incorporate variance inflation we use
\begin{align} \label{EKI_regul_vi_sub}
    \frac{\mathrm{d} u^{(j)}(t)}{\mathrm{d}t} &= (1-\rho)\left[-\widehat{C}^{u G}_t (G_{\bsi(t)}(u_t^{(j)}) - y_{\bsi(t)})-\widehat{C}^{u}_t C_0^{-1}u^{(j)}_t\right] \notag\\
    &+\rho\left[-\widehat{C}^{u G}_t (\bar G_{\bsi(t)}(u_t)- y_{\bsi(t)})-\widehat{C}^{u}_t C_0^{-1}\bar{u}_t\right] \qquad (j \in J)\\
    u(0) &= u_0, \notag
\end{align}
where $0\leq\rho< 1$. Here ${\bsi(t)}$ denotes an index process that determines which subset we are considering at which time points. The analysis of subsampling in continuous time as well as the definition of the index process will be introduced in the following subsection \ref{subsec:cont_time}, where we summarize the results of \cite{Latz2021}. Note that the empirical covariance in case of subsampling is given by

$$
\widehat{C}^{u G}_t := \frac{1}{\Nens-1} \sum_{j = 1}^\Nens (u^{(j)}(t) - \overline{u}(t)) \otimes (G_{\bsi(t)}(u^{(j)})(t) - \overline{G_{\bsi(t)}}(u)(t)),
$$

\subsection{Convergence analysis of EKI with subsampling} \label{subsec:cont_time}

To determine which subsample we use when solving the EKI we consider a continuous-time Markov process (CTMP) $\bsi: [0, \infty) \times \Omega \rightarrow I$ on $I$. This process is a piecewise constant process that randomly changes states at random times given by the distribution of $\Delta$. The CTMP has initial distribution $\bsi(0) \sim \mathrm{Unif}(I)$ and transition rate matrix
\begin{equation} 
    \label{eq_transition_rate_mat} A(t) := \frac{1}{(\Nsub-1)\eta(t)}\begin{pmatrix} 1 &  \cdots & 1 \\ \vdots & \ddots & \vdots \\ 1 & \cdots & 1 \end{pmatrix} - \frac{\Nsub}{(\Nsub-1)\eta(t)} \cdot \mathrm{id}_I, \qquad (t \geq 0)
\end{equation}
where $\eta: [0, \infty) \rightarrow (0, \infty)$ is the \emph{learning rate}, which is continuously differentiable and bounded from above.\\
Algorithm~\ref{algo} describes how we sample from $(\bsi(t))_{t \geq 0}$.

\begin{algorithm} 
\caption{Sampling $(\bsi(t))_{t \geq 0}$}\label{algo}
\begin{algorithmic}[1]
   \State draw $\bsi(0) \sim \mathrm{Unif}(I)$ and $t_0 \leftarrow  0$
      \State sample $\Delta$ with survival function $$\mathbb{P}(\Delta \geq t|t_0)  :=  \mathbf{1}[t <0 ] + \exp\left( - \int_{0}^t\eta(u + t_0)^{-1}\mathrm{d}u \right) \qquad (t \in [-\infty, \infty])$$
      \State set $\bsi|_{(t_0, t_0+ \Delta)} \leftarrow \bsi(t_0)$ 
      \State draw $\bsi(t_0 + \Delta) \sim \mathrm{Unif}(I \backslash \{\bsi(t_0)\})$
   \State increment $t_0 \leftarrow t_0 + \Delta$ and go to 2
\end{algorithmic}
\end{algorithm}

For a more detailed analysis of this particular CTMP $(\bsi(t))_{t \geq 0}$ we refer to \cite{Latz2021}. For other characterisations we refer the to \cite{Anderson1991,Gillespie}.\\

Next we define a stochastic approximation process.

\begin{definition} 
Let $(\mathbf{F}_i,i \in I): X \times [0, \infty) \rightarrow X$ be a family Lipschitz continuous functions. Then the tuple $(\bsi(t), u(t))_{t \geq 0}$ consisting of the family of flows $(\mathbf{F}_i)_{i \in I}$ and the index process $(\bsi(t))_{t \geq 0})$ that satisfies
\begin{align*}
    \dot{u}(t) &=  - \mathbf{F}_{\bsi(t)}(u(t),t) \qquad (t > 0) \\
    u(0) &= u_0 \in X,
\end{align*}
is defined as \emph{stochastic approximation process}.
\end{definition}

Furthermore, we define $\hPsi = \sum_{i \in I}\mathbf{F}_i/{\Nsub}$ and introduce the flow
\begin{align*}
    \dot{u}(t) &=  - \hPsi(u(t),t) \qquad (t > 0) \\
    u(0) &= u_0 \in X,
\end{align*}

To analyse the asymtotic behaviour of the process we need the following:

\begin{assumption} \label{Assum_conv} Let $d \in \mathbb{N}$ and $X := \mathbb{R}^d$. For any $i \in I$ assume:
\begin{itemize}
    \item[(i)] $\mathbf{F}_i \in C^1(X\times [0, \infty),X)$,
    \item[(ii)]there exists a measurable function $h: [0, \infty) \rightarrow \mathbb{R}$, with $\int_0^\infty h(t) \mathrm{d}t = \infty$ such that the flow $\varphi_t^{(i)}$ satisfies
    $$
     \langle \mathbf{F}_{i}(\varphi_t^{(i)}(u_0),t)- \mathbf{F}_{i}(\varphi_t^{(i)}(u_1),t), \varphi_t^{(i)}(u_0) - \varphi_t^{(i)}(u_1) \rangle_X \leq  -h(t) \| \varphi_t^{(i)}(u_0) - \varphi_t^{(i)}(u_1)\|^2
    $$
    for any two initial values $u_0, u_1 \in X$.
    
\end{itemize}
\end{assumption}
By Assumption~\ref{Assum_conv}(ii) we obtain that the flow of $-\hPsi$ is exponentially contracting. Hence, by the Banach fixed-point theorem, the flow has a unique stationary point $u^* \in X$. The main result of \cite{Latz2021} shows that the stochastic process converges to the unique stationary point $u^*$ of the flow $(\overline{\varphi}_t)_{t \geq 0}$ and is summarized below.

\begin{theorem}\label{thm_gen_Latz21} Consider the stochastic approximation process $(\bsi(t), u(t))_{t \geq 0}$, which is initialised with $(i_0, u_0) \in I \times X$. Furthermore, let Assumption~\ref{Assum_conv} hold and that the learning rate satisfies, $\lim_{t \rightarrow \infty}\eta(t) = 0$. Then
$$
\lim_{t \rightarrow \infty}\mathrm{d}_{\rm W}\left(\delta(\cdot - u^*), \mathbb{P}(u(t) \in \cdot | u_0, i_0)\right) = 0,
$$

where $\mathrm{d}_{\rm W}$ denotes the Wasserstein distance, i.e.
$$
\mathrm{d}_{\rm W}(\pi, \pi') = \inf_{H \in C(\pi, \pi')} \int_{X \times X} \min\{1, \|u - u'\|^q \} \mathrm{d}H(u, u'),
$$
where $q \in (0,1]$ and $C(\pi, \pi')$ denotes the set of couplings of the probability measures $\pi, \pi'$ on $(X, \mathcal{B}X)$.
\end{theorem}

For the proof we refer to \protect{\cite[Theorem A.1]{Hanu_2023}}.

To analyse convergence of our subsampling scheme we verify that the gradient flow satisfies Assumptions \ref{Assum_conv}. Then the result follows by theorem \ref{thm_gen_Latz21}. Condition $(ii)$ is hereby essential. We consider the scaled left hand side
\begin{align*}
-\frac{1}{\Nens}\langle u_1 -  u_2, \mathbf{F}_i(u_1,t) - \mathbf{F}_i(u_2,t) \rangle \leq -h(t) \| u_1 - u_2\|^2,
\end{align*}
for $t$ large enough, where $- \mathbf{F}_i(u(t),t)$ denotes the right hand side of the systems \eqref{EKI_regul_sub} and \eqref{EKI_regul_vi_sub} and $h: [0, \infty) \rightarrow \mathbb{R}$ being a measurable function. In order to derive convergence results in the parameter space, we will focus in the following on the regularised setting, i.e. we consider the potential $\Phi^{\scaleto{\mathrm{reg}}{5pt}}$ and only consider the variance inflated flow \eqref{EKI_regul_vi_sub}.

\begin{theorem}\label{thm:ss_novi}
Let Assumption \ref{assu:subpotentials_gradients} be satisfied and $\Nens>d$. Furthermore, let $(\bsi(t))_{t \geq 0})$ be an index process and assume $(u^{(j)}(t))_{t \geq 0, j \in J}$ satisfies \eqref{EKI_regul_sub} (or respectively \eqref{EKI_regul_vi_sub}), and $\alpha>2$ in theorem \ref{thm:conver_simon}. Then the stochastic approximation process $(\bsi(t), u^{(j)}(t))_{t \geq 0, j \in J}$ satisfies
$$
\lim_{t \rightarrow \infty}\mathrm{d}_{\rm W}\left(\delta(\cdot - u^*), \mathbb{P}(u^{(j)}(t) \in \cdot | u_0, i_0)\right) = 0 \qquad (j \in J).$$
\end{theorem}

\begin{proof}
    Note that we obtain \eqref{EKI_regul_sub} from \eqref{EKI_regul_vi_sub} by choosing $\rho=0$, thus, we will focus on \eqref{EKI_regul_vi_sub}.\\
    Let $u_1$ and $u_2$ be two coupled process with initial values $u_1(0),u_2(0).$ We want to show the existence of a function $h: [0, \infty) \rightarrow \mathbb{R}$ with $\int_0^\infty h(t) \mathrm{d}t = \infty$ such that
\begin{align*}
-\frac{1}{\Nens}\langle u_1 -  u_2, \mathbf{F}_i(u_1,t) - \mathbf{F}_i(u_2,t) \rangle \leq -h(t) \| u_1 - u_2\|^2,
\end{align*}

Note that due to the strong convexity of $\Phi^{\scaleto{\mathrm{reg}}{5pt}}$ we obtain from theorem \ref{thm:conver_simon} $u^{(j)}\to u^*$ with rate $\mathcal{O}(t^{-\frac{1}{\alpha}})$, i.e.

\[\|u^{(j)}_t-u^*\|\in \mathcal{O}(t^{-\frac{1}{\alpha}}) \quad \forall t\geq0 \quad \forall i\in\{1,...,\Nsub\}.\]

Furthermore, by theorem \ref{thm:conver_simon} we also obtain that $\Phi_i^{\scaleto{\mathrm{reg}}{5pt}}(u)$ is bounded, i.e., there exists a $B>0$ such that
\[\|\Phi_i^{\scaleto{\mathrm{reg}}{5pt}}(u)\|_2\leq B \quad \forall u\in\mathbb{R}^d.\]

Since $u_1$ and $u_2$ are column vectors consisting of the stacked particle vectors we need to introduce new variables to represent \eqref{EKI_regul_vi_sub} in vectorized notion.
We define for all $i\in\{1,...,\Nsub\}$ the operators $\mathcal{G}_i:X^J\rightarrow \mathbb{R}^{d\times J}$, $u\to \left[G_i(u),...,G_i(u)\right]^T\in\mathbb{R}^{k\Nens}$ and $\bar{\mathcal{G}_i}:X^J\rightarrow \mathbb{R}^{k\Nens}$, $u\to \left[\bar{G_i}(u),...,\bar{G_i}(u)\right]^T\in\mathbb{R}^{k\Nens}$, Moreover, we set $ \mathbf{\widehat C_u}=diag\{\widehat C_u,\widehat C_u,...,\widehat C_u\}\in\mathbb{R}^{d\Nens\times d\Nens}, \mathbf{\widehat C^{uG}_i}=diag\{\widehat C^{uG}_i,\widehat C^{uG}_i,...,\widehat C^{uG}_i\}\in\mathbb{R}^{d\Nens\times d\Nens},\mathbf{C_0^{-1}}=diag\{C_0^{-1},C_0^{-1},...,C_0^{-1}\}\in\mathbb{R}^{d\Nens\times d\Nens}$ and $\mathbf{y_i}=\left[y_i,y_i,\cdots,y_i\right]^T\in\mathbb{R}^{k\Nens}.$ For the potential we define

\[\mathbf{\Phi}_i(u)=\frac{1}{2} \|\mathbf{y_i} - \mathcal{G}_i(u)\|^2=\frac{1}{2}\sum_{j=1}^\Nens \|y_i - G_i(u^{(j)})\|^2.\]
Thus, we have
\begin{align}
    &-\frac{1}{\Nens}\langle u_1 -  u_2, \mathbf{F}_i(u_1,t) - \mathbf{F}_i(u_2,t) \rangle\notag\\
    =&-\frac{1-\rho}{\Nens}\langle u_1 -  u_2, \mathbf{\widehat C^{uG}_i} (\mathcal{G}_i(u_1)(t) - \mathbf{y_i})+\mathbf{\widehat{C}^{u_1}} \mathbf{C_0^{-1}}u_1(t) \rangle\label{proof:1}\\
    &+\frac{1-\rho}{\Nens}\langle u_1 -  u_2, \mathbf{\widehat C^{uG}_i} (\mathcal{G}_i(u_2)(t) - \mathbf{y_i})+\mathbf{\widehat{C}^{u_2}} \mathbf{C_0^{-1}}u_2(t) \rangle\label{proof:2}\\
    &-\frac{\rho}{\Nens}\langle u_1 -  u_2, \mathbf{\widehat C^{uG}_i} (\bar{\mathcal{G}_i}(u_1)(t) - \mathbf{y_i})+\mathbf{\widehat{C}^{u_1}} \mathbf{C_0^{-1}}u_1(t) \rangle\label{proof:3}\\
    &+\frac{\rho}{\Nens}\langle u_1 -  u_2, \mathbf{\widehat C^{uG}_i} (\bar{\mathcal{G}_i}(u_2)(t) - \mathbf{y_i})+\mathbf{\widehat{C}^{u_2}} \mathbf{C_0^{-1}}u_2(t) \rangle\label{proof:4}.
\end{align}
We will focus for now on \eqref{proof:1} and \eqref{proof:2}. Equations \eqref{proof:3} and \eqref{proof:4} can be analysed similarly.\\ We exploit in the following the fact that we can estimate the difference of the EKI flow to a preconditioned gradient flow. 
Adding $-\mathbf{\widehat{C}^{u_1}}\nabla\mathbf{\Phi}_i(u_1)+\mathbf{\widehat{C}^{u_1}}\nabla\mathbf{\Phi}_i(u_1)$ and 
$-\mathbf{\widehat{C}^{u_2}}\nabla\mathbf{\Phi}_i(u_2)+\mathbf{\widehat{C}^{u_2}}\nabla\mathbf{\Phi}_i(u_2)$ yields
\begin{align*}
 &-\frac{1}{\Nens}\langle u_1 -  u_2, \mathbf{F}_i(u_1,t) - \mathbf{F}_i(u_2,t) \rangle\\
    =&-\frac{1-\rho}{\Nens}\langle u_1 -  u_2, \mathbf{\widehat C^{uG}_i} (\mathcal{G}_i(u_1)(t) - \mathbf{y_i})-\mathbf{\widehat{C}^{u_1}}\nabla\mathbf{\Phi}_i(u_1)\rangle\\
    &-\frac{1-\rho}{\Nens}\langle u_1 -  u_2, \mathbf{\widehat{C}^{u_1}} \mathbf{C_0^{-1}}u_1(t)+\mathbf{\widehat{C}^{u_1}}\nabla\mathbf{\Phi}_i(u_1) \rangle\\
    &+\frac{1-\rho}{\Nens}\langle u_1 -  u_2, \mathbf{\widehat C^{uG}_i} (\mathcal{G}_i(u_2)(t) - \mathbf{y_i})-\mathbf{\widehat{C}^{u_2}}\nabla\mathbf{\Phi}_i(u_2)\rangle\\
    &+\frac{1-\rho}{\Nens}\langle u_1 -  u_2, \mathbf{\widehat{C}^{u_2}} \mathbf{C_0^{-1}}u_2(t)+\mathbf{\widehat{C}^{u_2}}\nabla\mathbf{\Phi}_i(u_2) \rangle.
\end{align*}
The first term satisfies
\begin{align*}
    &-\frac{1-\rho}{\Nens}\langle u_1 -  u_2, \mathbf{\widehat C^{uG}_i} (\mathcal{G}_i(u_1)(t) - \mathbf{y_i})-\mathbf{\widehat{C}^{u_1}}\nabla\mathbf{\Phi}_i(u_1)\rangle\\
    \leq & \frac{1-\rho}{\Nens}|\langle u_1 -  u_2, \mathbf{\widehat C^{uG}_i} (\mathcal{G}_i(u_1)(t) - \mathbf{y_i})-\mathbf{\widehat{C}^{u_1}}\nabla\mathbf{\Phi}_i(u_1)\rangle|\\    
    \leq & \frac{1-\rho}{\Nens}\|u_1 -  u_2\|_2 \|\mathbf{\widehat C^{uG}_i} (\mathcal{G}_i(u_1)(t) -  \mathbf{y_i})-\mathbf{\widehat{C}^{u_1}}\nabla\mathbf{\Phi}_i(u_1)\|_2
\end{align*}

Next we use a result from \protect{\cite[Lemma 4.5]{Weissmann2022}}. With this we can bound the second norm and obtain
\begin{align}  
    & \frac{1-\rho}{\Nens}\|u_1 -  u_2\|_2 \|\mathbf{\widehat C^{uG}_i} (\mathcal{G}_i(u_1)(t) - \mathbf{y_i})-\mathbf{\widehat{C}^{u_1}}\nabla\mathbf{\Phi}_i(u_1)\|_2\notag\\
    \leq &\frac{1-\rho}{\Nens}\sum_{j=1}^\Nens \|u_1^{(j)} -  u_2^{(j)}\|_2 \|\widehat C^{uG}_i (G_i(u_1^{(j)})(t) - y_i)-\widehat{C}^{u_1}\nabla\Phi_i(u_1^{(j)})\|_2\notag\\
    \leq &(1-\rho)\sum_{j=1}^\Nens \|u_1^{(j)} -  u_2^{(j)}\|_2 b_1 \sqrt{\Phi(u_1^{(j)})}V_{e1}(t)^{\frac{3}{2}}\in\mathcal{O}\left(t^{-\frac{3\alpha+2}{2\alpha}}\right)\label{proof:first_term}.
\end{align}
Respectively we can do the same for the third term.\\
For the second and fourth term we obtain
\begin{align*}
    &-\frac{1-\rho}{\Nens}\langle u_1 -  u_2, \mathbf{\widehat{C}^{u_1}} \mathbf{C_0^{-1}}u_1(t)+\mathbf{\widehat{C}^{u_1}}\nabla\mathbf{\Phi}_i(u_1) \rangle\\
    &+\frac{1-\rho}{\Nens}\langle u_1 -  u_2, \mathbf{\widehat{C}^{u_2}} \mathbf{C_0^{-1}}u_2(t)+\mathbf{\widehat{C}^{u_2}}\nabla\mathbf{\Phi}_i(u_2) \rangle\\
    =&-\frac{1-\rho}{\Nens}\langle u_1 -  u_2, \mathbf{\widehat{C}^{u_1}} \nabla\mathbf{\Phi}_i^{\scaleto{\mathrm{reg}}{5pt}}(u_1)\rangle+\frac{1-\rho}{\Nens}\langle u_1 -  u_2, \mathbf{\widehat{C}^{u_2}} \nabla\mathbf{\Phi}_i^{\scaleto{\mathrm{reg}}{5pt}}(u_2) \rangle\\
    =&-\frac{1-\rho}{\Nens}\langle u_1 -  u_2, \mathbf{\widehat{C}^{u_1}} \nabla\mathbf{\Phi}_i^{\scaleto{\mathrm{reg}}{5pt}}(u_1)-\mathbf{\widehat{C}^{u_2}} \nabla\mathbf{\Phi}_i^{\scaleto{\mathrm{reg}}{5pt}}(u_2)\rangle
\end{align*}
To bound the term
\begin{align*}
    -\frac{1-\rho}{\Nens}\langle u_1 -  u_2, \mathbf{\widehat{C}^{u_1}} \nabla\mathbf{\Phi}_i^{\scaleto{\mathrm{reg}}{5pt}}(u_1)-\mathbf{\widehat{C}^{u_2}} \nabla\mathbf{\Phi}_i^{\scaleto{\mathrm{reg}}{5pt}}(u_2)\rangle
\end{align*}
we take a mean-field approach \cite{doi:10.1137/21M1414000,doi:10.1137/19M1251655}, i.e. under suitable assumptions, the sample covariance has a well-defined limit $C(t)$, where $C(t)$ is symmetric, positive definite for all $t\ge 0$. Thus, by splitting
\begin{align*}
    &-\frac{1-\rho}{\Nens}\langle u_1 -  u_2, \mathbf{\widehat{C}^{u_1}} \nabla\mathbf{\Phi}_i^{\scaleto{\mathrm{reg}}{5pt}}(u_1)-\mathbf{\widehat{C}^{u_2}} \nabla\mathbf{\Phi}_i^{\scaleto{\mathrm{reg}}{5pt}}(u_2)\rangle\\&=  -\frac{1-\rho}{\Nens}\langle u_1 -  u_2, \mathbf{\widehat{C}^{u_1}} \nabla\mathbf{\Phi}_i^{\scaleto{\mathrm{reg}}{5pt}}(u_1)- \mathbf{\widehat{C}^{u_1}}  \nabla\mathbf{\Phi}_i^{\scaleto{\mathrm{reg}}{5pt}}(u_2)\rangle\\
    &+ -\frac{1-\rho}{\Nens}\langle u_1 -  u_2, \mathbf{\widehat{C}^{u_1}} (I-(\mathbf{\widehat{C}^{u_1}})^{-1}\mathbf{\widehat{C}^{u_2}})\nabla\mathbf{\Phi}_i^{\scaleto{\mathrm{reg}}{5pt}}(u_2)\rangle,
\end{align*}
the factor $(I-(\mathbf{\widehat{C}^{u_1}})^{-1}\mathbf{\widehat{C}^{u_2}})$ can be made arbitrarily small by adjusting $\Nens$. We therefore assume, that $\Nens$ is chosen large enough such that this term is negligible. Then by $\mu$ strong convexity we obtain the following upper bound for the first term
\begin{align*}
    &=  -\frac{1-\rho}{\Nens}\langle u_1 -  u_2, \mathbf{\widehat{C}^{u_1}} \nabla\mathbf{\Phi}_i^{\scaleto{\mathrm{reg}}{5pt}}(u_1)- \mathbf{\widehat{C}^{u_1}}  \nabla\mathbf{\Phi}_i^{\scaleto{\mathrm{reg}}{5pt}}(u_2)\rangle\\
    &\leq -\frac{1-\rho}{\Nens} \lambda_{min}(\widehat{C}^{u_1})\mu\|u_1 -  u_2\|^2 \in \mathcal{O}\left(t^{-\frac{2+\alpha}{\alpha}}\right),
\end{align*}
since $\lambda_{min}(\widehat{C}^{u_1})\in\mathcal{O}\left(t^{-1}\right)$ by Theorem \ref{thm:conver_simon}. Moreover, since $\alpha>2$, we have $-\frac{2+\alpha}{\alpha}>-\frac{3\alpha+2}{2\alpha}$ and therefore the term \eqref{proof:first_term} vanishes faster then the latter term. \\
The terms \eqref{proof:3} and \eqref{proof:4} can be bounded using the same arguments as well as \protect{\cite[Lemma 6.4]{Weissmann2022}}. The upper bound for these two together is
\[-\frac{\rho}{\Nens} \lambda_{min}(\widehat{C}^{u_1})\mu\|u_1 -  u_2\|^2 \in \mathcal{O}\left(t^{-\frac{2+\alpha}{\alpha}}\right).\]
All together we obtain
\begin{align*}
    -\frac{1}{\Nens}\langle u_1 -  u_2, \mathbf{F}_i(u_1,t) - \mathbf{F}_i(u_2,t) \rangle \leq -h(t) \| u_1 - u_2\|^2,
\end{align*}
where $h(t)=\frac{\mu}{\Nens}\lambda_{min}(\widehat{C}^{u_1})$ with $\int_0^\infty h(t) \mathrm{d}t = \infty$, since $\lambda_{min}(\widehat{C}^{u_1})\in\mathcal{O}\left(t^{-1}\right)$.

\end{proof}

\section{Numerical experiments}\label{sec:numeric}
We consider the image given in example \ref{ex:main} for the reconstruction and use the EKI \eqref{EKI_regul_vi} and the suggested EKI with subsampling scheme \eqref{EKI_regul_vi_sub} to estimate the unknown parameters (density $\rho$ and energy dissipation $E$ as discussed in subsection \ref{subsec:optim}), i.e. $d=2$. To ensure that the EKI is searching a solution in the full space $\mathbb{R}^d$ we consider $\Nens=3$. Furthermore, we represent the image as $705\times555$ pixel image, this means that the observation space is given by $Y=\mathbb{R}^{705\times555}$. We solve the ODEs \eqref{EKI_regul_vi} and \eqref{EKI_regul_vi_sub} up until time $T=10000$ with $\rho=0$ in \eqref{EKI_regul_vi} and respectively  \eqref{EKI_regul_vi_sub}. For the subsampling strategy a linear decaying learning rate $\eta(t)=(at+b)^{-1}$, where $a=b=10$ is considered. To account for the minimal step size of the ODE solver, the learning rate $\eta(t)$ only considered until time $T_{sub}=10$. Afterwards we consider a fixed amount of switching times. This results in approximately $600$ data switches until time $T_{sub}=10$. For the remaining time we switched the data $1000$ times. As illustrated in Figure \ref{fig:cartoon_subsamp} we split the data horizontally into $\Nsub=5$ subsets.

\begin{figure}
\centering
        \includegraphics[width=0.7\linewidth]{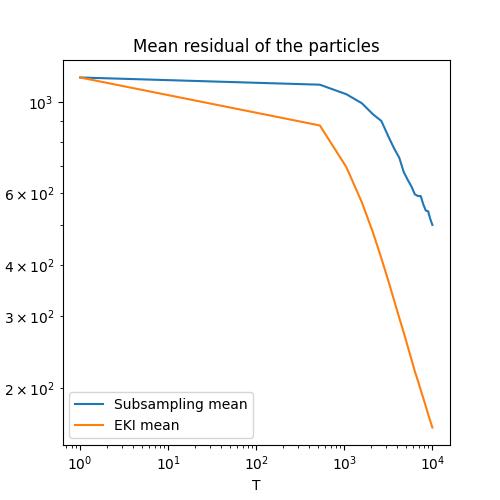}
\caption{Mean residuals of the particles, i.e. $\frac{1}{\Nens}\sum_{j=1}^\Nens \Phi^{\scaleto{\mathrm{reg}}{5pt}}(u^{(j)})$}
\label{fig:19_residuals_mean}
\end{figure}

Figure \ref{fig:19_residuals_mean} depicts the mean residuals of the solutions obtained by the EKI and our subsampling scheme. We can see that asymptotically both methods converge with the same rate. However, keep in mind the subsampling approach is computationally cheaper, since we only need to consider the lower-dimensional ODE \eqref{EKI_regul_vi_sub} at each point in time.

\begin{figure}[H]
        \includegraphics[width=0.49\linewidth]{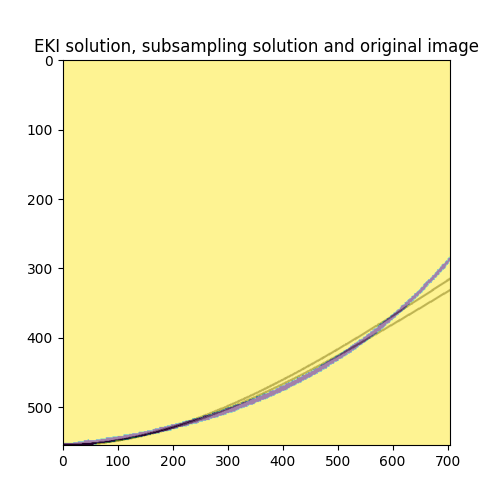}
        \includegraphics[width=0.49\linewidth]{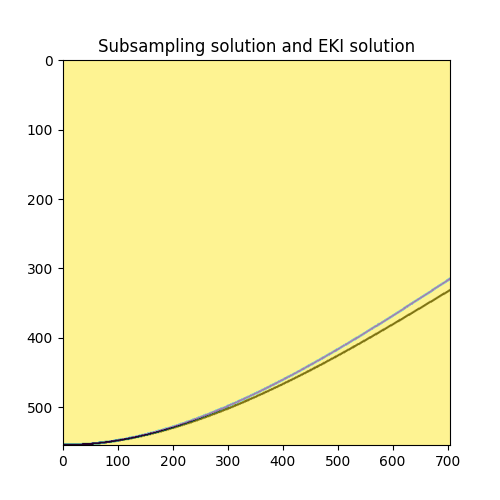}

\caption{Comparison of the best computed solutions to the original image. The blue rod depicts the rod from the original image. The thinner rods depict the solutions computed via the EKI (upper rod) and our subsampling approach (lower rod). The right image illustrates the EKI solution (upper rod) and the subsampling solution (lower rod). The $x$ and $y$ axis represent the pixels of the image.}
\label{fig:19_comparision of computed methods}
\end{figure}

Figure \ref{fig:19_comparision of computed methods} depicts the computed solutions in comparison to the original image. We can see that the solutions computed from both methods approximate the rod in the original image quite well. One can see that the desired solution seems to be more elastic than the solutions computed via the EKI and our subsampling approach, which is due to model error. We observe that the subsampling strategy leads to a good solution while reducing the computational costs significantly.

\section{Conclusions} \label{sec:conclusion}
We have introduced a formulation of a guide wire system with estimation of the unknown parameters by a subsampling version of EKI using high resolution images as data. The experiment with real data shows promising results; the subsampling strategy could achieve a good accuracy while reducing the computational costs significantly. In future work, we will explore this direction further to enable uncertainty quantification in the state estimation, thus making a robust control of the guide wire possible. 

\section*{Data availability}
Data for the conducted numerical experiments is available in the repository: \url{https://github.com/matei1996/EKI_subsampling_guidewire}

\section*{Acknowledgements}
MH and CS are grateful for the support from MATH+ project EF1-19: Machine Learning Enhanced Filtering Methods for Inverse Problems and EF1-20: Uncertainty Quantification and Design of Experiment for Data-Driven Control, funded by the Deutsche Forschungsgemeinschaft (DFG, German Research
Foundation) under Germany's Excellence Strategy – The Berlin Mathematics
Research Center MATH+ (EXC-2046/1, project ID: 390685689). The authors of this publication also thank the University of Heidelberg's Excellence Strategy (Field of Focus 2) for the funding of the project "Ariadne- A new approach for automated catheter control" that allowed the research and development of the control system of the catheter.

\section*{Authors' contributions}
GK designed the forward model and planned the numerical experiments. JM and JS carried out the experimental setup, designing an electrical drive system as well generating data for the numerical experiments. JH designed the data segmentation in the forword model and helped to draft the manuscript. CS introduced and analysed the subsampling scheme for the EKI as well as draft the manuscript. MH introduced and analysed the subsampling scheme for the EKI, carried out the numerical experiments as well as draft the manuscript. All authors read and approved the final manuscript.

\section*{Competing Interests}
The authors declare that they have no competing interests.

\bibliographystyle{siam}
\bibliography{main}   

\end{document}